\documentclass[10pt, conference, letterpaper]{ieeeconf}
\IEEEoverridecommandlockouts

\title{\LARGE \bf On data-driven control: informativity of noisy input-output data with cross-covariance bounds
}

\author{Tom R.V. Steentjes, Mircea Lazar, Paul M.J. Van den Hof
\thanks{T.R.V. Steentjes, M. Lazar and P.M.J. Van den Hof are with the Department of Electrical Engineering, Eindhoven University of Technology, The Netherlands. E-mails: \texttt{\{t.r.v.steentjes, m.lazar, p.m.j.vandenhof\}@tue.nl}}
\thanks{This project has received funding from the European Research Council (ERC), Advanced Research Grant SYSDYNET, under the European Union's Horizon 2020 research and innovation programme (grant No. 694504).}
}

\usepackage{amsmath}
\usepackage{amssymb}
\usepackage{mathrsfs}
\usepackage{graphicx}
\usepackage{theorem}
\usepackage[caption=false, font=footnotesize]{subfig}
\usepackage{xcolor}
\usepackage{cite}
\usepackage{stfloats}
\usepackage[hyphens]{url}
\usepackage{hyperref}

{\theorembodyfont{\slshape}\newtheorem{thm}{Theorem}}
{\theorembodyfont{\slshape}}
{\theorembodyfont{\slshape}\newtheorem{lemma}{Lemma}}
{\theorembodyfont{\upshape}\newtheorem{definition}{Definition}}
{\theorembodyfont{\slshape}}
{\theorembodyfont{\upshape}\newtheorem{remark}{Remark}}
{\theorembodyfont{\slshape}\newtheorem{proposition}{Proposition}}
{\theorembodyfont{\upshape}}
{\theorembodyfont{\upshape}}

\begin{document}
\maketitle
\thispagestyle{empty}
\begin{abstract}
In this paper we develop new data informativity based controller synthesis methods that extend existing frameworks in two relevant directions: a more general noise characterization in terms of cross-covariance bounds and informativity conditions for control based on input-output data. Previous works have derived necessary and sufficient informativity conditions for noisy input-state data with quadratic noise bounds via an S-procedure. Although these bounds do not capture cross-covariance bounds in general, we show that the S-procedure is still applicable for obtaining non-conservative conditions on the data. Informativity-conditions for stability, $\mathcal{H}_\infty$ and $\mathcal{H}_2$ control are developed, which are sufficient for input-output data and also necessary for input-state data. Simulation experiments illustrate that cross-covariance bounds can be less conservative for informativity, compared to norm bounds typically employed in the literature.
\end{abstract}

\section{Introduction}
When mathematical models of dynamical systems are not available, data plays an essential role in the process of learning system characteristics. Indeed, data can contain information about the system from which a model of the system can be derived or a controller can be learned, either from a data-based model or directly from the data. A key problem for data-driven control is to determine whether a set of data collected from a system contains enough information to design a controller, independent of the methodology.

An indirect approach for controller design from data consists of two steps: obtaining a model from data through system identification \cite{ljung1999} and subsequently designing a controller via a model-based method. In the field of identification for control, the problem of determining a suitable model for controller design is considered~\cite{vandenhofetal1995}, aiming at minimizing performance degradation due to model mismatching. If the data used for obtaining a model are sufficiently rich for identification, is determined by a property called informativity.

Even if data are not informative for identification, data can still be informative for controller design. Necessary and sufficient conditions for informativity of data for control were developed in \cite{waarde2020tac} for noiseless input-state data. These results were extended in \cite{vanwaarde2020noisytac} for noisy input-state data with prior knowledge on the noise in the form of quadratic bounds, via a matrix variant of the S-procedure. Quadratic noise bounds play a key role in data-driven controller design~\cite{vanwaarde2020noisytac},~\cite{persis2020},~\cite{berberich2021}, distributed controller design~\cite{steentjes2021cdc} and dissipativity analysis~\cite{koch2020},~\cite{vanwaarde2021dis} from data, and represent, for example, magnitude, energy and variance bounds on the noise.

In this paper, we consider the problem of determining informativity of \emph{input-output} and input-state data for control with prior knowledge of the noise in the form of a sample \emph{cross-covariance} type bound with respect to a user-chosen instrumental signal. Bounds on the sample cross-covariance were introduced in \cite{hakvoort95} as an alternative to magnitude bounds in parameter bounding identification, given its overly conservative noise characterization, cf. \cite{bisoffi2021} for a comparison of instantaneous and (quadratic) energy type bounds for data-driven control. Our approach to data-driven control extends existing frameworks in two relevant directions: a more general noise characterization in terms of cross-covariance bounds with practical relevance and informativity conditions for control based on input-output data. We provide sufficient conditions for informativity for stabilization, $\mathcal{H}_\infty$ and $\mathcal{H}_2$ control, which are also necessary for input-state data.
\section{Input-output data: cross-covariance bounds}
Consider a class of linear systems described by
\begin{align} \label{eq:iosys}
A(q^{-1})y(t)=B(q^{-1})u(t)+e(t),
\end{align}
with $A(\xi)\in\mathbb{R}^{p\times p}[\xi]$ and $B(\xi)\in\mathbb{R}^{p\times m}[\xi]$ polynomial matrices, given by $A(\xi)=I+A_1\xi +A_2\xi^2+\cdots +A_n\xi^l$ and $B(\xi)=B_0+B_1\xi+B_2\xi^2+\cdots +B_l\xi^l$ and $q^{-1}$ is the delay operator, i.e., $q^{-1}x(t)=x(t-1)$. By defining $\zeta(t):=\operatorname{col}(y(t-1),\dots, y(t-l),u(t-1),\dots, u(t-l))$, $\bar{A}:=\operatorname{row}(-A_1,\dots, -A_l)$ and $\bar{B}:=\operatorname{row}(B_1,\dots, B_l)$, we can write~\eqref{eq:iosys} equivalently as
\begin{align}
&y(t)=B_0u(t)+e(t)+\begin{bmatrix}
\bar{A} & \bar{B}
\end{bmatrix}\zeta(t), \label{eq:iosyszo}
\end{align}
Hence, with $\zeta\in\mathbb{R}^{n}$ as a state, a state-space representation~is
\begin{align} 
\zeta(t+1)\!&=\!\!\underbrace{\left[\begin{array}{cc|cc}
\multicolumn{2}{c|}{\bar{A}}  & \multicolumn{2}{c}{\bar{B}} \\ \hline I & 0 & 0 & 0\\ 0 & 0 & 0 & 0\\ 0 & 0 & I & 0
\end{array}\right]}_{=:A_z}\!\!\zeta(t)+\!\!\underbrace{\left[\begin{array}{c}
B_0\\ \hline 0\\ I \\ 0
\end{array}\right]}_{=:B_z}\!\!u(t)+\!\!\underbrace{\left[\begin{array}{c}
I\\ \hline 0\\ 0\\ 0
\end{array}\right]}_{=:H_z}\!\!e(t),\nonumber\\
y(t)&=\begin{bmatrix}
\bar{A} & \bar{B}
\end{bmatrix}\zeta(t)+B_0u(t)+e(t),\label{eq:iosysz}
\end{align}
Notice that \eqref{eq:iosysz} is a non-minimal representation of order $n=(p+m)l$. Defining the data matrices $Z_-:=\begin{bmatrix}
\zeta(0) & \cdots & \zeta(N-1)
\end{bmatrix}$, $Y_-:=\begin{bmatrix}
y(0) & \cdots & y(N-1)
\end{bmatrix}$ and $U_-$, $E_-$ accordingly, we obtain the data equation
\begin{align} \label{eq:iodateq}
Y_-=\begin{bmatrix}
\bar{A} & \bar{B}\end{bmatrix} Z_-+B_0U_-+E_-,
\end{align}
where $\bar{A}$, $\bar{B}$, $B_0$ are unknown system matrices. We consider the noise not to be measured, i.e., $E_-$ is unknown, while prior knowledge on the cross-covariance of the noise with respect to an instrumental variable is available.
\subsection{Cross-covariance noise bounds}
Consider the sample cross-covariance with respect to the noise $e\in\mathbb{R}^p$ and a variable $r\in\mathbb{R}^M$, given by $\frac{1}{\sqrt{N}}\sum_{t=0}^{N-1}e(t)r(t)^\top=\frac{1}{\sqrt{N}}E_-R_-^\top$. The variable $r$ is instrumental and can be specified by the user (as discussed at the end of this subsection), i.e., it is a given variable in the upcoming analysis. We assume prior knowledge on the noise of the form
\begin{equation} \label{eq:qccbnd}
\frac{1}{N}E_-R_-^\top R_-E_-^\top\preceq H_u, \vspace{-0.15em}
\end{equation}
where $H_u$ is an upper-bound on the squared sample cross-covariance matrix $\frac{1}{\sqrt{N}}E_-R_-^\top$. In generalized form, we write
\vspace{-0.3em}
\begin{equation} \label{eq:iogenqccbnd}
\begin{bmatrix}
I\\ R_-E_-^\top
\end{bmatrix}^\top \begin{bmatrix}
Q_{11} & Q_{12}\\ Q_{12}^\top & Q_{22}
\end{bmatrix}\begin{bmatrix}
I\\ R_-E_-^\top
\end{bmatrix}\succeq 0, \vspace{-0.2em}
\end{equation} 
with $Q_{22}\prec 0$. For $Q_{11}=NH_u$, $Q_{12}=0$ and $Q_{22}=-I$, the bound \eqref{eq:qccbnd} is recovered. Note that~\eqref{eq:iogenqccbnd} can be rewritten as the bound in \cite[Assumption~1]{vanwaarde2020noisytac} with $\Phi_{22} = R_-^\top Q_{22} R_-$, but in general only $\Phi_{22} \preceq 0$ holds, while $\Phi_{22} \prec 0$ is assumed in~\cite{vanwaarde2020noisytac}. This means that the data informativity results of~\cite{vanwaarde2020noisytac} cannot be used to establish data informativity for general cross-covariance noise bounds. However, the matrix S-lemma in~\cite[Theorem~13]{vanwaarde2020noisytac} can still be exploited to obtain necessary conditions for informativity of input-state data with cross-covariance bounds, as shown in Proposition 1 of this paper.

In the state-space representation \eqref{eq:iosysz}, the matrices $A_z$ and $B_z$ contain unknown parameters $\bar{A}$, $\bar{B}$ and $B_0$. Write
\begin{align}
&A_z=\underbrace{\left[\begin{array}{c|c}
\bar{A} & \bar{B}\\ \hline 0 & 0\\ 0 & 0\\ 0 &0
\end{array}\right]}_{=:\Lambda_e}+\underbrace{\left[\begin{array}{cc|cc}
0 & 0 & 0 & 0\\ \hline I_{p(l-1)} & 0 & 0 & 0\\ 0 & 0 & 0& 0\\ 0 & 0 &I_{m(l-1)} & 0
\end{array}\right]}_{=:J_1}\label{eq:ioA}\\
&\text{and } B_z=\underbrace{\left[\begin{array}{c}
B_0\\ \hline 0\\ 0 \\ 0
\end{array}\right]}_{=:B_e} +\underbrace{\left[\begin{array}{c}
0\\ \hline 0\\ I_m \\ 0
\end{array}\right]}_{=:J_2},  \label{eq:ioB}
\end{align}
so that $\Lambda_e$ and $B_e$ are unknown parameter matrices, concatenated with zero rows, and $J_1$ and $J_2$ are binary matrices. The set of all pairs $(\Lambda_e,B_e)$ that are compatible with the data is
\[\Sigma_{(U,Y)}^{RQ}\!:=\!\{\left[\begin{smallmatrix} \bar{A} &\bar{B}\\ 0 & 0\end{smallmatrix}\right]\!,\!\left[\begin{smallmatrix} B_0\\ 0\end{smallmatrix}\right])\,|\, \exists E_- \text{ such that } \eqref{eq:iogenqccbnd} \text{ and }\eqref{eq:iodateq} \text{ hold}\}. \vspace*{-2em}\]
\begin{lemma} \label{lem:ioparam}
Let $Q_e:=\left[\begin{smallmatrix}
H_zQ_{11}H_z^\top & H_zQ_{12}\\ Q_{12}^\top H_z^\top & Q_{22}
\end{smallmatrix}\right]$ and consider
\begin{equation}\label{eq:ioqab}
\begin{bmatrix}
I\\ \Lambda_e^\top\\B_e^\top
\end{bmatrix}^\top\begin{bmatrix}
I & H_zY_-R_-^\top\\ 0 & -Z_-R_-^\top\\ 0 & -U_-R_-^\top
\end{bmatrix}Q_e\begin{bmatrix}
I & H_zY_-R_-^\top\\ 0 & -Z_-R_-^\top\\ 0 & -U_-R_-^\top
\end{bmatrix}^\top\begin{bmatrix}
I\\ \Lambda_e^\top\\B_e^\top
\end{bmatrix}\succeq 0.
\end{equation}
It holds that $\Sigma_{(U,Y)}^{RQ}\!\subseteq\! \{(\Lambda_e,B_e)\,|\, \eqref{eq:ioqab}\text{ holds}\}$. Moreover, if $R_-[Z_-^\top\ U_-^\top]$ has full column rank, then $\Sigma_{(U,Y)}^{RQ}\!\!\!\!=\!\{(\Lambda_e,B_e)\,|\, \eqref{eq:ioqab} \text{ holds}\}$.
\end{lemma}
\begin{proof}
The first statement follows from~\eqref{eq:iodateq}, \eqref{eq:iogenqccbnd}, and the definition of $\Sigma_{(U,Y)}^{RQ}$. Full-column rank of $R_-[Z_-^\top\ U_-^\top]$ and \eqref{eq:ioqab} imply that the last $ml+p(l-1)$ rows of $[\Lambda_e\ B_e]=\operatorname{col}(M_1,0)$ are zero and $E_-:=Y_- -M_1[Z_-^\top\ U_-^\top]$ satisfies~\eqref{eq:iogenqccbnd}. Hence, $(\Lambda_e,B_e)\in\Sigma_{(U,Y)}^{RQ}$ so that $\Sigma_{(U,Y)}^{RQ}=\{(\Lambda_e,B_e)\,|\, \eqref{eq:ioqab} \text{ holds}\}$.
\end{proof}
We denote all $(A_z,B_z)$ that are compatible with the data by
\[\bar{\Sigma}_{(U,Y)}^{RQ}:=\{(\Lambda_e+J_1,B_e+J_2)\,|\, (\Lambda_e,B_e)\in\Sigma_{(U,Y)}^{RQ}\}.\]
We have provided a parametrization of (a superset of) $\Sigma_{(U,Y)}^{RQ}$ based on the data equation \eqref{eq:iodateq}. One can equivalently parametrize $\bar{\Sigma}_{(U,Y)}^{RQ}$ on the basis of the state data equation $Z_+=A_zZ_-+B_zU_-+H_zE_-$. This leads to an equal set $\bar{\Sigma}_{(U,Y)}^{RQ}$, but the `repeated' data in the parametrization contained in $Z_+:=\begin{bmatrix}\zeta(1) & \cdots & \zeta(N)\end{bmatrix}$, would render the evaluation numerically sensitive. Design methods in~\cite{vanwaarde2020noisytac}, \cite{berberich2021} can yield a parametrization $\bar{\Sigma}_{(U,Y)}^{RQ}$ based on this state data equation, but with limited applicability to cross-covariance bounds~\eqref{eq:iogenqccbnd}, i.e., only if the dimension of $r$ satisfies $M\geq N$.

Existing guidelines~\cite{hakvoort95} recommend choosing an instrumental variable $r$ that is correlated with the input $u$, but uncorrelated with the noise $e$. Hence, this suggests the choice of (filtered/delayed versions) of the input for $r$ in an open-loop case for data collection, and an external reference/dithering signal for $r$ in a closed-loop case. Moreover, Lemma~\ref{lem:ioparam} provides an additional guideline for the choice of $r$ to reduce conservatism in the case of input-output data, i.e., $R_-[Z_-^\top\ U_-^\top]$ has full column rank only if the number of instrumental variables $M$ satisfies $M \geq pl+m(l+1)$.
\subsection{Output-feedback control}
Consider a (dynamic) output feedback controller described by the difference equation of the form \cite{persis2020}
\begin{align} \label{eq:ioc}
C(q^{-1})u(t)=D(q^{-1})y(t),
\end{align}
with $C(\xi)\in\mathbb{R}^{m\times m}[\xi]$ and $D(\xi)\in\mathbb{R}^{m\times p}[\xi]$ polynomial matrices given by $C(\xi)=I+C_1\xi+C_2\xi^2+\cdots + C_n\xi^l$ and
$D(\xi)=D_1\xi+D_2\xi^2+\cdots +D_n\xi^l$. We define a state $\zeta_c$ for \eqref{eq:ioc} as $\zeta_c:=\operatorname{col} (u(t-1),\dots,u(t-l),y(t-1),\dots,y(t-l))$, yielding a state-space representation for the controller:
\begin{align}
\zeta_c(t+1)&=\left[\begin{array}{cc|cc}
\multicolumn{2}{c|}{\bar{C}}  & \multicolumn{2}{c}{\bar{D}} \\ \hline I & 0 & 0 & 0\\ 0 & 0 & 0 & 0\\ 0 & 0 & I & 0
\end{array}\right]\zeta_c(t)+\left[\begin{array}{c}
0\\ \hline 0\\ I \\ 0
\end{array}\right]y(t),\label{eq:iosyszc}\\
u(t)&=\begin{bmatrix}
\bar{C} & \bar{D}
\end{bmatrix}\zeta_c(t),\nonumber
\end{align}
with $\bar{C}:=\operatorname{row}(-C_1,\dots,-C_l)$ and $\bar{D}:=\operatorname{row}(D_1,\dots,D_l)$. It follows that $\zeta_c=\left[\begin{smallmatrix}
0 & I\\ I & 0
\end{smallmatrix}\right]\zeta$, which implies that $u(t)=\begin{bmatrix}
\bar{C} & \bar{D}
\end{bmatrix}\zeta_c(t)=\begin{bmatrix}
\bar{D} & \bar{C}
\end{bmatrix}\zeta(t)$.
Hence, the closed-loop system described by \eqref{eq:iosysz} and \eqref{eq:iosyszc} has a representation
\begin{align} \label{eq:iosyszcl}
\zeta(t+1)=\underbrace{\left[\begin{array}{cc|cc}
\multicolumn{2}{c|}{\bar{A}+B_0\bar{D}} & \multicolumn{2}{c}{\bar{B}+B_0\bar{C}}\\
\hline I & 0 & 0 & 0\\
\hline \multicolumn{2}{c|}{\bar{D}} & \multicolumn{2}{c}{\bar{C}}\\
\hline 0 & 0 & I & 0
\end{array}\right]}_{=:A_{\text{cl}}}\zeta(t)+\left[\begin{array}{c}
I\\ \hline 0\\ 0 \\0
\end{array}\right]e(t).
\raisetag{14pt} 
\end{align}
With $K:=\begin{bmatrix}
\bar{D} & \bar{C}
\end{bmatrix}$, the closed-loop system matrix $A_{\text{cl}}$ satisfies $A_\text{cl}=A_z+B_zK$. For some $(A_z,B_z)\in \bar{\Sigma}_{(U,Y)}^{RQ}$, we say that the controller~\eqref{eq:ioc} stabilizes~\eqref{eq:iosys} if the closed-loop system~\eqref{eq:iosyszcl} is stable, i.e., if all eigenvalues of $A_z+B_zK$ are in the open unit disk, since this implies stability of the closed-loop system~\eqref{eq:iosys} and \eqref{eq:ioc}. The notion of stabilization with respect to the state-space representation~\eqref{eq:iosysz} was introduced in~\cite{persis2020} for data-driven stabilization. We note that in the single-input-single-output case, $(A_z,B_z)$ is controllable if and only if $A(\xi)$ and $B(\xi)$ are coprime~\cite{persis2020}.

\section{Informativity for stabilization}
\subsection{Informativity of input-output data}
\begin{definition}
The data $(U,Y)$ are said to be informative for quadratic stabilization by output-feedback controller~\eqref{eq:ioc} if there exist a $K$ and $P\succ 0$ so that
\begin{align*}
\bar{\Sigma}_{(U,Y)}^{RQ}\subseteq \{(A,B)\,|\, (A+BK) P (A+BK)^\top-P\prec 0\}.
\end{align*}
\end{definition}
By \eqref{eq:ioA}-\eqref{eq:ioB}, we find that the existence of $K$ and $P\succ 0$ so that $(A_z+B_zK) P(A_z+B_zK)^\top-P\prec 0$, is equivalent to the existence of $K$ and $P\succ 0$ such that~\eqref{eq:ioqstab} holds true. Now, for the data $(U,Y)$ to be informative for quadratic stabilization, we require the existence of $K$ and $P\succ 0$ so that \eqref{eq:ioqstab} holds for all $(\Lambda_e,B_e)\in\Sigma_{(U,Y)}^{RQ}$. This is precisely a problem that can be solved by the S-procedure; more specifically, by the matrix-valued S-lemma \cite{vanwaarde2020noisytac}.
\begin{figure*}[b]
\vspace{-1em}
\hrulefill
\begin{align} \label{eq:ioqstab}
\begin{bmatrix}
I\\ \Lambda_e^\top\\B_e^\top
\end{bmatrix}^\top \underbrace{\begin{bmatrix}
P-(J_1+J_2K)P(J_1+J_2K)^\top & -(J_1+J_2K)P & -(J_1+J_2K)PK^\top\\
-P(J_1+J_2K)^\top & -P & -PK^\top\\
-KP(J_1+J_2K)^\top & -KP & -KPK^\top
\end{bmatrix}}_{=:\Pi}\begin{bmatrix}
I\\ \Lambda_e^\top\\B_e^\top
\end{bmatrix}\succ 0
\end{align}
\end{figure*}
\vspace*{-.5em}
\begin{thm} \label{thm:ioqstab}
The data $(U,Y)$ are informative for quadratic stabilization by output feedback controller \eqref{eq:ioc} if there exist $L\in\mathbb{R}^{m\times n}$, $P\succ 0$, $\alpha\geq 0$ and $\beta>0$ so that \eqref{eq:ioinfqstabLMI} holds true. Moreover, for $L$ and $P$ such that \eqref{eq:ioinfqstabLMI} is satisfied, $A_\text{cl}=A_z+\allowbreak B_zK$ is stable for all $(A_z,B_z)\in\bar{\Sigma}_{(U_,Y)}^{RQ}$ with $K:=LP^{-1}$.
\end{thm}
\vspace*{-.5em}
\begin{proof}
Let $L$, $P\succ 0$, $\alpha\geq 0$ and $\beta>0$ exist so that~\eqref{eq:ioinfqstabLMI} holds true and consider the matrix $\Pi$ defined in~\eqref{eq:ioqstab}. By the Schur complement, \eqref{eq:ioinfqstabLMI} is equivalent to
\begin{align*}
\Pi-\alpha \Lambda &\succeq \begin{bmatrix}
\beta I & 0\\ 0 & 0
\end{bmatrix},\quad \text{where } K:=LP^{-1}\quad \text{and}\\
 \Lambda&:=\begin{bmatrix}
I & H_zY_-R_-^\top\\  0 & -Z_-R_-^\top\\ 0 & -U_-R_-^\top
\end{bmatrix}Q_e\begin{bmatrix}
I & H_zY_-R_-^\top\\ 0 & -Z_-R_-^\top\\ 0 & -U_-R_-^\top
\end{bmatrix}^\top.
\end{align*}
Hence, \eqref{eq:ioqstab} holds true for all $(\Lambda_e,B_e)\in\Sigma_{(U,Y)}^{RQ}$, cf. \cite[Theorem~13]{vanwaarde2020noisytac}. This concludes the proof.
\end{proof}

We remark that if there is a $Z$ so that $\bar{Z}^\top \Lambda\bar{Z}\succ 0$ with $\bar{Z}:=\operatorname{col}(I,Z)$, called the generalized Slater condition \cite{vanwaarde2020noisytac}, then~\eqref{eq:ioinfqstabLMI} is also a necessary condition for informativity of input-output data for quadratic stabilization, if $R_-[Z_-^\top\ U_-^\top]$ has full column rank (Lemma~\ref{lem:ioparam}). Unlike in the case of input-state data, which will be discussed next, we note that the generalized Slater condition can in general not hold true in the input-output case if $l\geq 1$, since the noise affects a subspace of the extended state space, yielding a degenerate matrix $\Lambda$. The combination of noisy and noiseless states in $\zeta$ suggests that necessity could potentially be proven in general by a `fusion' of the matrix S-lemma and matrix Finsler's lemma~\cite{vanwaarde2021finsler}.

\begin{figure*}[b]
\vspace{-1em}
    \hrulefill
\begin{align} \label{eq:ioinfqstabLMI}
\begin{bmatrix}
P-\beta I & -J_1P-J_2L & 0 & J_1P+J_2L\\
-PJ_1^\top-L^\top J_2^\top & -P & -L^\top & 0\\
0& -L & 0 & L\\
PJ_1^\top+L^\top J_2^\top & 0 & L^\top & P
\end{bmatrix}-\alpha \begin{bmatrix}
I & H_zY_-R_-^\top\\  0 & -Z_-R_-^\top\\ 0 & -U_-R_-^\top\\ 0 & 0
\end{bmatrix}Q_e\begin{bmatrix}
I & H_zY_-R_-^\top\\ 0 & -Z_-R_-^\top\\ 0 & -U_-R_-^\top\\ 0 & 0
\end{bmatrix}^\top\succeq 0
\end{align}
\end{figure*}

\subsection{Informativity of input-state data} \label{sec:inputstate}
We will now consider a special case, where input-state data is available instead of input output data. That is, we measure a state $y(t)=x(t)$ and the class of systems considered is
\begin{align} \label{eq:sys}
x(t+1)=Ax(t)+Bu(t)+e(t),
\end{align}
with the corresponding data equation
\begin{align} \label{eq:dateq}
X_+=AX_-+B U_-+E_-.
\end{align}

All systems that explain the data $(U_-,X)$ for some $E_-$ satisfying the cross-covariance bound \eqref{eq:iogenqccbnd} are in the set
\begin{equation*}
\Sigma_{(U_-,X)}^{RQ}:=\{(A,B)\,|\, \exists E_- \text{ such that } \eqref{eq:iogenqccbnd} \text{ and } \eqref{eq:dateq} \text{ hold}\}.
\end{equation*}
By \eqref{eq:dateq}, the set of feasible systems is $\Sigma_{(U_-,X)}^{RQ}=\{(A,B)\,|\, (A,B) \text{ satisfies } \eqref{eq:qab}\}$, where
\begin{align}\label{eq:qab}
\begin{bmatrix}
I\\ A^\top\\B^\top
\end{bmatrix}^\top\underbrace{\begin{bmatrix}
I & X_+R_-^\top\\ 0 & -X_-R_-^\top\\ 0 & -U_-R_-^\top
\end{bmatrix}Q\begin{bmatrix}
I & X_+R_-^\top\\ 0 & -X_-R_-^\top\\ 0 & -U_-R_-^\top
\end{bmatrix}^\top}_{=:\Lambda_X}\begin{bmatrix}
I\\ A^\top\\B^\top
\end{bmatrix}\succeq 0.\raisetag{12pt}
\end{align}
\vspace*{-2em}
\begin{remark}
Consider a specific selection of $M=N$ instrumental variables defined by $r_i(t):=\delta(t-i+1)$, $ i=1,\dots,N$, and $\delta:\mathbb{Z}\to \{0,1\}$ is the unit impulse defined as $\delta(0)=1$ and $\delta(x)=0$ for $x\in\mathbb{Z}\setminus\{0\}$. It follows that $R_-=I$ for this choice of instrumental signals. Then, with the generalized quadratic cross-covariance bound~\eqref{eq:iogenqccbnd}, we observe that for this special choice $R_-=I$, we recover the set of feasible systems in~\cite{vanwaarde2020noisytac}, and, hence, the informativity conditions in~\cite{vanwaarde2020noisytac}.
\end{remark}
\vspace{-1.5em}
\begin{definition}
The data $(U_-,X)$ are said to be informative for quadratic stabilization by state feedback if there exist a feedback gain $K$ and $P\succ 0$ so that
\begin{equation*}
\Sigma_{(U_-,X)}^{RQ}\subseteq \{(A,B)\,|\, (A+BK)P (A+BK)^\top-P\prec 0\}. \vspace{-0.25em}
\end{equation*}
\end{definition}

We will now provide a necessary and sufficient condition for informativity of input-state data for quadratic stabilization, given prior knowledge on the cross-covariance \eqref{eq:iogenqccbnd}. Consider the generalized Slater condition
\begin{equation} \label{eq:genslatx}
\begin{bmatrix}
I\\Z
\end{bmatrix}^\top \Lambda_X\begin{bmatrix}
I\\ Z
\end{bmatrix}\succ 0. \vspace{-0.25em}
\end{equation}

\begin{proposition}
Suppose that there exists a $Z$ so that~\eqref{eq:genslatx} holds true. Then the data $(U_-,X)$ are informative for quadratic stabilization if and only if there exist $L\in\mathbb{R}^{m\times n}$, $P\succ 0$, $\alpha\geq 0$ and $\beta>0$ so that
\begin{align} \label{eq:infqstabLMI}
\begin{bmatrix}
P-\beta I & 0 & 0 & 0\\ 0 & -P & -L^\top & 0\\ 0 & -L & 0 & L\\ 0 & 0 & L^\top & P
\end{bmatrix}-\alpha \begin{bmatrix}
\Lambda_X & 0\\ 0 & 0
\end{bmatrix}\succeq 0.
\end{align}
Moreover, $K$ is such that $A+BK$ is stable for all $(A,B)\in\Sigma_{(U_-,X)}^{RQ}$ if $K:=LP^{-1}$ with $L$ and $P\succ 0$ satisfying \eqref{eq:infqstabLMI}.
\end{proposition}
\begin{proof}
$(\Leftarrow)$ This is proven by the same argument as in the proof of Theorem~\ref{thm:ioqstab}. $(\Rightarrow)$ Let the data be informative for quadratic stabilization, i.e., there exist $K$ and $P\succ 0$ so that, with $\Pi$ defined in \eqref{eq:ioqstab} with $J_1=0$, $J_2=0$:
\begin{align*}
&\begin{bmatrix}
I\\ A^\top\\ B^\top
\end{bmatrix}^\top\!\! \Pi\, (\star)\succ 0 \text{ for all } (A,B) \text{ with } \begin{bmatrix}
I\\ A^\top\\ B^\top
\end{bmatrix}^\top \!\! \Lambda_X\, (\star)\succeq 0,\\
&\text{where } \Lambda_X=\left[\begin{array}{c|c}
\Lambda_{11}^X & \Lambda_{12}^X\\ \hline \Lambda_{21}^X & \Lambda_{22}^X
\end{array}\right] := \left[\begin{array}{cc}
I & X_+R_-^\top\\ \hline 0 & -X_-R_-^\top\\ 0 & -U_-R_-^\top
\end{array}\right]Q\ (\star)^\top.
\end{align*}
We will now show that $\operatorname{ker}\Lambda_{22}^X\subseteq \Lambda_{12}^X$, such that necessity follows by the matrix S-lemma \cite{vanwaarde2020noisytac}. First, notice that $\operatorname{ker} \Lambda_{22}^X = \operatorname{ker} R_-\begin{bmatrix}
X_-^\top & U_-^\top
\end{bmatrix}$. Now, take any $x\in\operatorname{ker} \Lambda_{22}^X$. Then $R_-\begin{bmatrix}
X_-^\top & U_-^\top
\end{bmatrix} x=0$. Clearly, we have that $(X_+R_-^\top Q_{22}+Q_{12})R_-\begin{bmatrix}
X_-^\top & U_-^\top \end{bmatrix} x=0$, which implies that $x\in \operatorname{ker} \Lambda_{12}^X$. Since $x\in\operatorname{ker} \Lambda_{22}^X$ was chosen arbitrary, this shows that $\operatorname{ker}\Lambda_{22}^X\subseteq \Lambda_{12}^X$. By $\operatorname{ker}\Lambda_{22}^X\subseteq \Lambda_{12}^X$ and \eqref{eq:genslatx}, there exist $\alpha\geq 0$ and $\beta>0$ so that, by \cite[Theorem~13]{vanwaarde2020noisytac}:
\begin{align*}
\Pi-\alpha \Lambda_X\succeq \begin{bmatrix}
\beta I & 0\\ 0 & 0
\end{bmatrix},
\end{align*}
which is equivalent to~\eqref{eq:infqstabLMI} for $L:=KP$ by the Schur complement. This completes the proof.
\end{proof}

\section{Including performance specifications}
We will now consider the problem of finding a controller~\eqref{eq:ioc} for which the closed-loop system achieves an $\mathcal{H}_\infty$ or $\mathcal{H}_2$ performance bound from the input-output data $(U,Y)$. Consider the performance output $z$, given by $z(t)=C_z\zeta(t)+D_zu(t)$. For any pair $(A_z,B_z)$, the controller~\eqref{eq:ioc} yields the closed loop system
\begin{align*}
\zeta(t+1)&=(A_z+B_zK)\zeta(t)+H_ze(t),\\
z(t)&=(C+DK)\zeta(t).
\end{align*}
Hence, the transfer matrix from $e$ to $z$ is given by
\begin{align*}
T(z) := (C_z+D_zK)(zI-A_z-B_zK)^{-1}H_z,
\end{align*}
for which the $\mathcal{H}_\infty$ and $\mathcal{H}_2$ norm are denoted $\|T\|_{\mathcal{H}_\infty}$ and $\|T\|_{\mathcal{H}_2}$, respectively.

For given $K$, the $\mathcal{H}_\infty$ norm of $T$ is less than $\gamma$, $\|T\|_{\mathcal{H}_\infty}<\gamma$, if and only if there exists $X\succ 0$ such that~\cite[p. 125]{schererweilandLMI} 
\begin{align} \label{eq:Hinf1}
\begin{bmatrix}
X & 0 & A_K^\top X & C_K^\top\\ 0 & \gamma I & H_z^\top X & 0\\ XA_K & X H_z & X & 0\\ C_K & 0 & 0 & \gamma I
\end{bmatrix}\succ 0,
\end{align}
where $A_K:=A_z+B_zK$ and $C_K:=C_z+D_zK$.
\begin{definition}
The data $(U,Y)$ are said to be informative for common $\mathcal{H}_\infty$ control by output-feedback controller \eqref{eq:ioc} with performance $\gamma$ if there exist a $K$ and $X\succ 0$ so that 
\begin{align*}
\bar{\Sigma}_{(U,Y)}^{RQ}\subseteq \{(A_z,B_z)\,|\,\eqref{eq:Hinf1} \text{ holds true} \}.
\end{align*}
\end{definition}
\vspace{-1em}

\begin{figure*}[b]
\vspace{-1em}
    \hrulefill
\begin{align}
\begin{bmatrix}
P-\gamma^{-1}H_zH_z^\top-\beta I & 0 & 0 & J_1P+J_2L & 0\\
0 & 0 & 0 & P & 0\\
0 & 0 & 0 & L & 0\\
P^\top J_1^\top+L^\top J_2^\top & P & L^\top & P & F^\top\\
0 & 0 & 0 & F & \gamma I
\end{bmatrix}-\alpha \!\begin{bmatrix}
I & H_zY_-R_-^\top\\  0 & -Z_-R_-^\top\\ 0 & -U_-R_-^\top\\ 0 & 0\\ 0 & 0
\end{bmatrix}\!Q_e\!\begin{bmatrix}
I & H_zY_-R_-^\top\\ 0 & -Z_-R_-^\top\\ 0 & -U_-R_-^\top\\ 0 & 0\\ 0 & 0
\end{bmatrix}^\top\!\!\!\!\! \succeq 0,\ \begin{bmatrix}
P & F^\top\\ F & \gamma I
\end{bmatrix}\succ 0 \label{eq:ioinfHinf}
\end{align}
\end{figure*}

\begin{thm}
The data $(U,Y)$ are informative for common $\mathcal{H}_\infty$ control with performance $\gamma$  if there exist $L\in\mathbb{R}^{m\times n}$, $P\succ 0$, $\alpha\geq 0$ and $\beta>0$ so that \eqref{eq:ioinfHinf} holds true.
\end{thm}
\begin{proof}
By a congruence transformation of \eqref{eq:Hinf1} with $\operatorname{diag}(P,I,P,I)$ with $P:=X^{-1}$ and the application of the Schur complement (twice), the existence of $K$ and $X\succ 0$ so that \eqref{eq:Hinf1} holds, is equivalent to the existence of $P$ and $L$ so that $P\succ 0$ and
\begin{align} \label{eq:Hinf2}
P-V_z\underbrace{(P-\gamma^{-1} F^\top F)^{-1}}_{=:S}V_z^\top-\gamma^{-1} H_zH_z^\top\succ 0
\end{align} 
and $P-\gamma^{-1}F^\top F \succ 0$, where $V_z:=A_zP+B_zL$ and $F:=C_zP+D_zL$. We can now rewrite \eqref{eq:Hinf2} as
\begin{align*}
\begin{bmatrix}
I\\ A_z^\top\\ B_z^\top
\end{bmatrix}^\top\begin{bmatrix}
P-\gamma^{-1} H_zH_z^\top & 0\\ 0 & -\begin{bmatrix}
P\\ L
\end{bmatrix}S\begin{bmatrix}
P\\ L
\end{bmatrix}^\top
\end{bmatrix}\begin{bmatrix}
I\\ A_z^\top\\ B_z^\top
\end{bmatrix}\succ 0,
\end{align*}
which is equivalent to
\begin{align}
&\begin{bmatrix}
I\\ \Lambda_e^\top\\ B_e^\top
\end{bmatrix}^\top\!\!\!\Pi_{\mathcal{H}_\infty}\!\!\begin{bmatrix}
I\\ \Lambda_e^\top\\ B_e^\top
\end{bmatrix}:=\begin{bmatrix}
I\\ \Lambda_e^\top\\ B_e^\top
\end{bmatrix}^\top\!\!\!\begin{bmatrix}
P-\gamma^{-1}H_zH_z^\top & 0\\ 0 & 0
\end{bmatrix}\!\!\begin{bmatrix}
I\\ \Lambda_e^\top\\ B_e^\top
\end{bmatrix}\nonumber\\
&-\!\begin{bmatrix}
I\\ \Lambda_e^\top\\ B_e^\top
\end{bmatrix}^\top\!\!\begin{bmatrix}
J_1 P+J_2L\\ P\\ L
\end{bmatrix}\!S\ (\star)^\top\!\! \begin{bmatrix}
I\\ \Lambda_e^\top\\ B_e^\top
\end{bmatrix}\!\succ \!0.\label{eq:Hinf3}
\end{align}
Hence, the data $(U,Y)$ are informative for common $\mathcal{H}_\infty$ control with performance $\gamma$ if and only if there exist $P\succ 0$ and $L$ such that $P-\gamma^{-1}F^\top F \succ 0$ and \eqref{eq:Hinf3} holds for all $(\Lambda_e,B_e)\in \Sigma_{(U,Y)}^{RQ}$. By assumption, there exist $P\succ 0$, $L$, $\alpha\geq 0$ and $\beta >0$ such that \eqref{eq:ioinfHinf} holds true. By the Schur complement, \eqref{eq:ioinfHinf} is equivalent to $\Pi_{\mathcal{H}_\infty}-\alpha \Lambda \succeq \left[\begin{smallmatrix}
\beta I & 0\\ 0 & 0
\end{smallmatrix}\right]$, which implies that \eqref{eq:Hinf3} holds for all $(\Lambda_e,B_e)\in \Sigma_{(U,Y)}^{RQ}$.
\end{proof}

The conditions \eqref{eq:ioinfHinf} are linear with respect to $P$, $L$, $\alpha$ and $\beta$. By a straightforward additional application of the Schur complement, \eqref{eq:ioinfHinf} can also be made linear with respect to $\gamma$.

For a given controller parameter matrix $K$, the $\mathcal{H}_2$ norm of $T$ is less than $\gamma$, $\|T\|_{\mathcal{H}_2}<\gamma$, if and only if there exists $X\succ 0$ such that~\cite[Proposition~II.1]{steentjes2020arxiv}, cf.~\cite{vanwaarde2020noisytac}
\begin{align} \label{eq:H21}
\operatorname{trace}H_z^\top X H_z <\gamma^2\ \text{and}\quad X\! \succ\! A_K^\top\! XA_K+C_K^\top C_K.
\end{align}
\vspace*{-2.5em}
\begin{definition}
The data $(U,Y)$ are said to be informative for common $\mathcal{H}_2$ control by output-feedback controller \eqref{eq:ioc} with performance $\gamma$ if there exist a $K$ and $X\succ 0$ so that $\bar{\Sigma}_{(U,Y)}^{RQ}\subseteq \{(A_z,B_z)\,|\,\eqref{eq:H21} \text{ holds true} \}$.
\end{definition}
\vspace*{-1em}
\begin{figure*}[t]
\begin{align}
\begin{bmatrix}
P-\beta I & 0 & 0 & J_1P+J_2L & 0\\
0 & 0 & 0 & P & 0\\
0 & 0 & 0 & L & 0\\
P^\top J_1^\top+L^\top J_2^\top & P & L^\top & P & F^\top\\
0 & 0 & 0 & F & I
\end{bmatrix}-\alpha \begin{bmatrix}
I & H_zY_-R_-^\top\\  0 & -Z_-R_-^\top\\ 0 & -U_-R_-^\top\\ 0 & 0\\ 0 & 0
\end{bmatrix}Q_e\begin{bmatrix}
I & H_zY_-R_-^\top\\ 0 & -Z_-R_-^\top\\ 0 & -U_-R_-^\top\\ 0 & 0\\ 0 & 0
\end{bmatrix}^\top &\succeq 0\label{eq:ioinfH2}
\end{align}
\hrulefill \vspace{-1em}
\end{figure*}

\begin{thm}
The data $(U,Y)$ are informative for common $\mathcal{H}_2$ control with performance $\gamma$  if there exist $L\in\mathbb{R}^{m\times n}$, symmetric $Z$, $P\succ 0$, $\alpha\geq 0$ and $\beta>0$ so that $\operatorname{trace} Z <\gamma^2$, \eqref{eq:ioinfH2} holds true,
\begin{align} \label{eq:ioinfH2p}
\begin{bmatrix}
P & F^\top\\ F & I
\end{bmatrix}\succ 0\quad \text{ and }\quad \begin{bmatrix}
Z & H_z^\top\\ H_z & P
\end{bmatrix}\succeq 0.
\end{align}
\end{thm}
\begin{proof}
By a congruence transformation of \eqref{eq:H21} with $P:=X^{-1}$ and the Schur complement (in both directions), it follows that \eqref{eq:H21} is equivalent to $P-F^\top F\succ 0$,
\begin{align} \label{eq:H22}
P-V_z(P-F^\top F)^{-1}V_z^\top\succ 0
\end{align}
and $\operatorname{trace} H_z^\top P^{-1} H_z <\gamma^2$.
Now, we can rewrite \eqref{eq:H22} as
\begin{align*}
\begin{bmatrix}
I\\ A_z^\top\\ B_z^\top
\end{bmatrix}^\top \begin{bmatrix}
P & 0\\ 0 & -\begin{bmatrix}
P\\ L
\end{bmatrix}(P-F^\top F)^{-1}\begin{bmatrix}
P\\ L
\end{bmatrix}^\top
\end{bmatrix}\begin{bmatrix}
I\\ A_z^\top\\ B_z^\top
\end{bmatrix}\succ 0,
\end{align*}
which, by~\eqref{eq:ioA}-\eqref{eq:ioB}, holds if and and only if
\begin{align} \label{eq:H23}
\begin{bmatrix}
I\\ \Lambda_e^\top\\ B_e^\top
\end{bmatrix}^\top\!\!\!\begin{bmatrix}
J_1 P+J_2L\\ P\\ L
\end{bmatrix}\!S(\star)^\top\!\! \begin{bmatrix}
I\\ \Lambda_e^\top\\ B_e^\top
\end{bmatrix}\!\!\prec\!\! \begin{bmatrix}
I\\ \Lambda_e^\top\\ B_e^\top
\end{bmatrix}^\top\!\!\!\!\begin{bmatrix}
P & 0\\ 0 & 0
\end{bmatrix}\!\!\begin{bmatrix}
I\\ \Lambda_e^\top\\ B_e^\top
\end{bmatrix}\!.
\end{align}
There exist $P\succ 0$, $L$ so that \eqref{eq:H23}, $P-F^\top F\succ 0$ and $\operatorname{trace} H_z^\top P^{-1} H_z <\gamma^2$ if and only if there exist $P\succ 0$, $L$, $Z$ so that \eqref{eq:H23}, $P-F^\top F\succ 0$, $Z-H_z^\top P^{-1}H_z\succeq 0$ and $\operatorname{trace} Z <\gamma^2$. Indeed, for $Z:=H_z^\top P^{-1} H_z$ we infer $\operatorname{trace} Z<\gamma^2$. Sufficiency follows from $H_z^\top P^{-1} H_z\preceq Z\ \Rightarrow\ \operatorname{trace}H_z^\top P^{-1} H_z\leq  \operatorname{trace}Z$. Hence, the data $(U,Y)$ are informative for common $\mathcal{H}_2$ control with performance $\gamma$ if and only if there exist $P\succ 0$, $L$, $Z$ so that $Z-H_z^\top P^{-1}H_z\succeq 0$, $\operatorname{trace} Z <\gamma^2$, $P-F^\top F \succ 0$ and \eqref{eq:H23} hold for all $(\Lambda_e,B_e)\in \Sigma_{(U,Y)}^{RQ}$. By assumption, $\operatorname{trace} Z <\gamma^2$ is satisfied, $P-F^\top F \succ 0$, $Z-H_z^\top P^{-1}H_z\succeq 0$ follow by \eqref{eq:ioinfH2p} and via an analogue argument as in the proof of Theorem~2, \eqref{eq:H23} holds for all $(\Lambda_e,B_e)\in \Sigma_{(U,Y)}^{RQ}$ by \eqref{eq:ioinfH2}.
\end{proof}
\vspace*{-0.5em}
\begin{remark}
The conditions in Theorem~2/3 are also necessary for informativity of input-state data for $\mathcal{H}_\infty$/$\mathcal{H}_2$ control, where $H_z=I$, $J_1=0$, $J_2=0$ and $Y_-$ and $Z_-$ are replaced by $X_+$ and $X_-$, if \eqref{eq:genslatx} holds for some $Z$.
\end{remark}

\section{Numerical example}
Consider the system \eqref{eq:sys} with true system matrices
\begin{align*}
A_0=\begin{bmatrix}
-0.2414 &  -0.8649 &   0.6277\\
    0.3192 &  -0.0301  &  1.0933\\
    0.3129  & -0.1649   & 1.1093
\end{bmatrix},\quad B_0= \begin{bmatrix}
1 & 0\\
    0 & 2\\
    1 & 1
\end{bmatrix}.
\end{align*}
and consider a performance output $z(t)=[0\ 0\ 1]x(t)$. The objective is to determine if input state data collected from the system are informative for common $\mathcal{H}_2$ control. We consider a noise signal $e(t)$ with a uniform distribution, taking values from the closed ball $\{e\in \mathbb{R}^3\,|\, \|e\|_2^2\leq 0.35\}$. First, we consider this noise bound to be known, represented by the noise model $E_-\in \{E_-\,|\, E_-E_-^\top \preceq 0.35NI\}$ as described in \cite[Section VI.A]{vanwaarde2020noisytac}. This can be represented by the noise model~\eqref{eq:qccbnd} with $R_-=I$, cf.~\cite[Equation (5)]{vanwaarde2020noisytac}. We consider the informativity analysis for various data lengths $N$ ranging from $N=2$ to $N=250$. For each data length $N$, we generate $50$ data sets. Given the bound on $E_-$, we can verify informativity for common $\mathcal{H}_2$ control via Theorem~17 in \cite{vanwaarde2020noisytac}. We find that the generalized Slater condition~\cite[Equation (16)]{vanwaarde2020noisytac}, holds true for all data sets, thus the data are informative for common $\mathcal{H}_2$ control with performance $\gamma$ if and only if the condition \cite[Equation~($\mathcal{H}_2$)]{vanwaarde2020noisytac} is feasible. The relative number of data sets for which this necessary and sufficient condition is feasible for some $\gamma>0$ is visualized in Figure~\ref{fig:nodatfeasH2} for each data length $N$, in red. Naturally, if the condition is not feasible for any $\gamma>0$, the data are actually not informative for feedback stabilization, although the true system \emph{is} stable.

\begin{figure}[!t]
\vspace*{-0.8em}
      \centering
      \subfloat[Informativity for $\mathcal{H}_2$ control\label{fig:nodatfeasH2}]{%
       \includegraphics[scale=0.425]{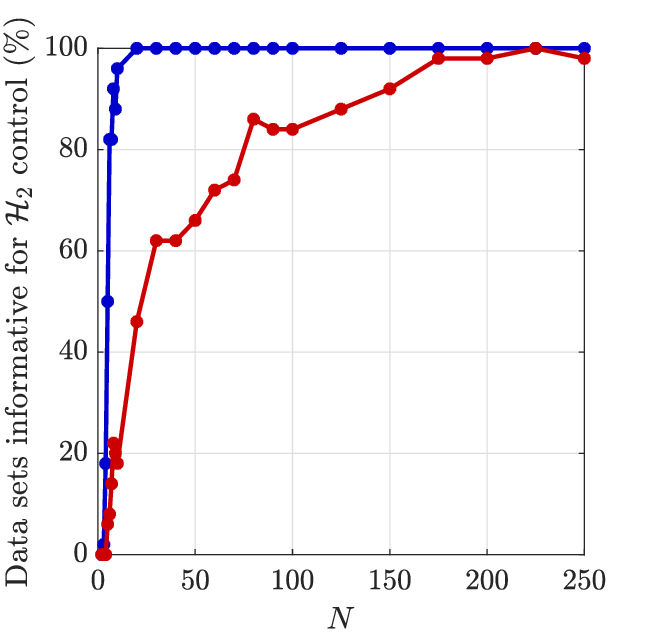}}
    \subfloat[$\mathcal{H}_2$ performance from data\label{fig:normfeasH2}]{%
       \includegraphics[scale=0.425]{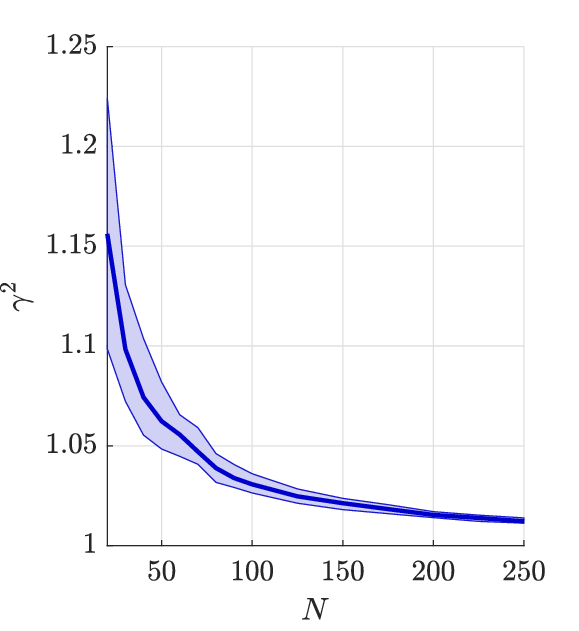}}
      \caption{(a) Number of input-state data sets that are informative for $\mathcal{H}_2$ control versus data length $N$ for noise-norm bounds (red) and quadratic cross-covariance bounds (blue) and (b) feasible $\gamma^2$ obtained versus data length $N$ with quadratic cross-covariance bounds.}\label{fig:infH2}\vspace*{-1.2em}
    \end{figure}

Now, we consider the quadratic cross-covariance bound~\eqref{eq:qccbnd} for the noise. We choose an instrumental variable that contains lagged versions of the input:
\[r(t):=\operatorname{col}(u(t),u(t-1),u(t-2),\dots, u(t-8),u(t-9)). \vspace{-0.5em}\]
We assume prior knowledge in the sense that $E_-\in\mathcal{E}_{RQ}=\{E_-\,|\, E_-R_-^\top R_-E_-^\top\preceq N H_u\}$, where $H_u$ is taken as $H_u=I$, independent of $N$. The cross-covariance bounds hold true for all generated data sets.  We verify that there exists some $Z$ so that \eqref{eq:genslatx} holds true for all data sets. Hence, by Remark 2, the data are informative for common $\mathcal{H}_2$ control with performance $\gamma$ if and only if the conditions in Theorem~3 are feasible. The relative number of data sets for which this necessary and sufficient condition is feasible for some $\gamma>0$ is visualized in Figure~\ref{fig:nodatfeasH2} for each data length $N$, in blue. For $N\geq 20$, all data sets are informative for common $\mathcal{H}_2$ control. For these data sets, the smallest $\mathcal{H}_2$ norm upper bounds $\gamma^2$ are visualized in Figure~\ref{fig:normfeasH2}, where the median performance is indicated by a solid line and the shaded area is bounded by the 25\textsuperscript{th} and 75\textsuperscript{th} percentiles. In comparison, the $\mathcal{H}_2$ norm that can be achieved by a state feedback controller with knowledge of $(A_0,B_0)$ is equal to $1.000$, which therefore is a benchmark that cannot be outperformed by any data-based controller.

Now, consider that noisy output measurements are available instead of state measurements. Consider system \eqref{eq:iosys} with $A(q^{-1})$ and $B(q^{-1})$ such that $T_0(q^{-1})=A^{-1}(q^{-1})B(q^{-1})$ with $T_0:=C_0(qI-A_0)^{-1}B_0$, where $C_0$ is the output matrix. We consider three cases: $C_0=[1\ 0\ 1]$, $C_0=[0\ 1\ 0]$, and $C_0=[1\ 0\ 0]$. The noise is uniformly drawn from $[-0.35,\, 0.35]$. For each choice of output, we generate 50 data sets for data lengths ranging from $N=2$ to $N=250$. We choose an instrumental signal containing lagged input signals as before, which is therefore independent on the choice of output. The upper-bound is chosen $H_u=0.3$, which holds for all data sets. By Theorem~3, feasibility of the conditions for informativity for $\mathcal{H}_2$ control for some $\gamma>0$ is verified for each data set. The results are depicted in Figure~\ref{fig:IOnodatfeasH2}. We observe that the data sets are not informative for low data lengths, which can be expected. For increasing data length, informativity becomes dependent on the choice of output. For $N=30$, for example, $90\%$ of the data sets yielded feasible informativity conditions for the choice of $C_0=[1\ 0\ 0]$, compared to less than $50\%$ of the data sets for the other two choices for $C_0$.

\begin{figure}[!t]
\centering
\includegraphics[scale=.425]{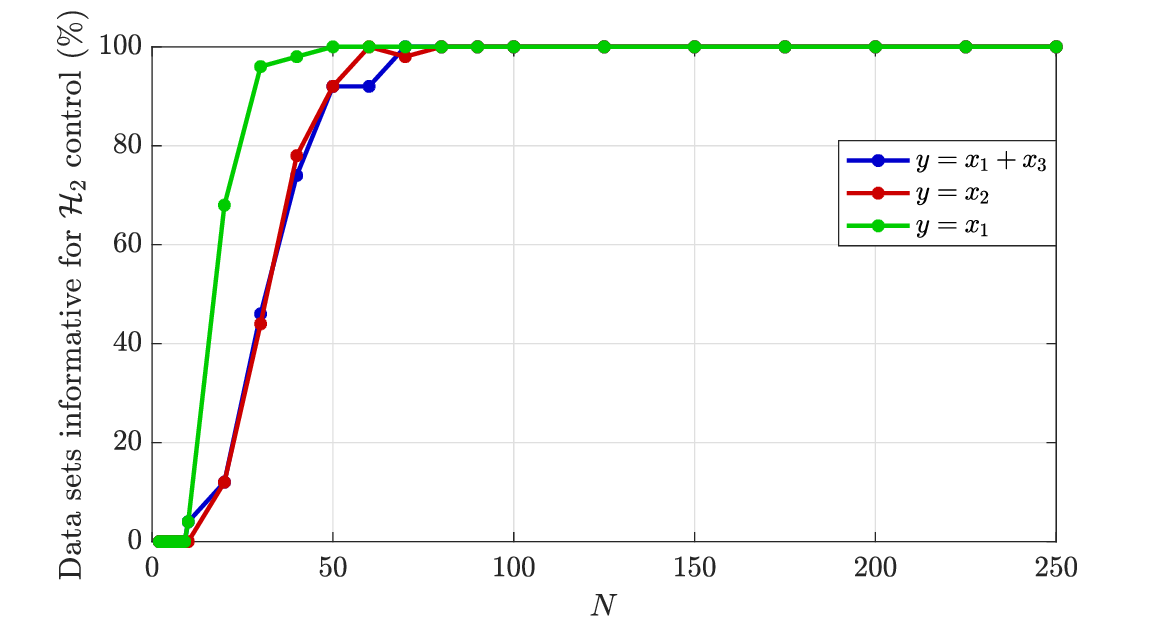}
\caption{Effect of the choice of output on informativity of input-output data for $\mathcal{H}_2$ control.}
\label{fig:IOnodatfeasH2} \vspace*{-1em}
\end{figure}
\section{Concluding remarks}
We have considered the problem of informativity of input-output data for control, with prior knowledge of the noise in the form of quadratic sample cross-covariance bounds. Sufficient informativity conditions for stabilization, $\mathcal{H}_\infty$ and $\mathcal{H}_2$ control via dynamic output feedback were derived, which are also necessary if  the state is measured. We have provided a numerical case study where data-informativity can be concluded with cross-covariance bounds, while the data are concluded to be non-informative with magnitude bounds. Finally, we have illustrated how the choice of output affects the informativity of input-output data via a numerical example.
\vspace{-0.5em}
\bibliographystyle{IEEEtran}
\bibliography{../rfrncs21a}

\begin{thebibliography}{10}
\providecommand{\url}[1]{#1}
\csname url@samestyle\endcsname
\providecommand{\newblock}{\relax}
\providecommand{\bibinfo}[2]{#2}
\providecommand{\BIBentrySTDinterwordspacing}{\spaceskip=0pt\relax}
\providecommand{\BIBentryALTinterwordstretchfactor}{4}
\providecommand{\BIBentryALTinterwordspacing}{\spaceskip=\fontdimen2\font plus
\BIBentryALTinterwordstretchfactor\fontdimen3\font minus
  \fontdimen4\font\relax}
\providecommand{\BIBforeignlanguage}[2]{{%
\expandafter\ifx\csname l@#1\endcsname\relax
\typeout{** WARNING: IEEEtran.bst: No hyphenation pattern has been}%
\typeout{** loaded for the language `#1'. Using the pattern for}%
\typeout{** the default language instead.}%
\else
\language=\csname l@#1\endcsname
\fi
#2}}
\providecommand{\BIBdecl}{\relax}
\BIBdecl

\bibitem{ljung1999}
L.~Ljung, \emph{System Identification: Theory for the User}.\hskip 1em plus
  0.5em minus 0.4em\relax Upper Saddle River, NJ, USA: Prentice Hall PTR, 1999.

\bibitem{vandenhofetal1995}
P.~M.~J. {Van den Hof} and R.~J.~P. {Schrama}, ``Identification and control:
  Closed-loop issues,'' \emph{Automatica}, vol.~31, no.~12, pp. 1751 -- 1770,
  1995.

\bibitem{waarde2020tac}
H.~J. {van Waarde}, J.~{Eising}, H.~L. {Trentelman}, and M.~K. {Camlibel},
  ``Data informativity: A new perspective on data-driven analysis and
  control,'' \emph{IEEE Trans. Autom. Control}, vol.~65, no.~11, pp.
  4753--4768, 2020.

\bibitem{vanwaarde2020noisytac}
H.~J. van Waarde, M.~K. Camlibel, and M.~Mesbahi, ``From noisy data to feedback
  controllers: Nonconservative design via a matrix {S}-lemma,'' \emph{IEEE
  Trans. Autom. Control}, vol.~67, no.~1, pp. 162--175, 2022.

\bibitem{persis2020}
C.~{De Persis} and P.~{Tesi}, ``Formulas for data-driven control:
  Stabilization, optimality, and robustness,'' \emph{IEEE Trans. Autom.
  Control}, vol.~65, no.~3, pp. 909--924, 2020.

\bibitem{berberich2021}
J.~Berberich, C.~W. Scherer, and F.~Allg\"{o}wer, ``Combining prior knowledge
  and data for robust controller design,'' \emph{ArXiv e-prints}, vol.
  arXiv:2009.05253, 2021.

\bibitem{steentjes2021cdc}
T.~R.~V. Steentjes, M.~Lazar, and P.~M.~J. {Van den Hof},
  ``{$\mathcal{H}_\infty$} performance analysis and distributed controller
  synthesis for interconnected linear systems from noisy input-state data,'' in
  \emph{Proc. 60\textsuperscript{th} IEEE Conf. Decis. Control (CDC)}, Austin,
  TX, USA, 2021, pp. 3717--3722.

\bibitem{koch2020}
A.~{Koch}, J.~{Berberich}, and F.~{Allg\"{o}wer}, ``Verifying dissipativity
  properties from noise-corrupted input-state data,'' in \emph{Proc.
  59\textsuperscript{th} IEEE Conf. Decis. Control (CDC)}, Jeju, South Korea,
  2020, pp. 616--621.

\bibitem{vanwaarde2021dis}
H.~J. van Waarde, M.~K. Camlibel, P.~Rapisarda, and H.~L. Trentelman,
  ``Data-driven dissipativity analysis: application of the matrix {S}-lemma,''
  \emph{ArXiv e-prints}, vol. arXiv:2109.02090, 2021.

\bibitem{hakvoort95}
R.~G. Hakvoort and P.~M.~J. {Van den Hof}, ``Consistent parameter bounding
  identification for linearly parametrized model sets,'' \emph{Automatica},
  vol.~31, no.~7, pp. 957--969, 1995.

\bibitem{bisoffi2021}
A.~Bisoffi, C.~{De Persis}, and P.~Tesi, ``Trade-offs in learning controllers
  from noisy data,'' \emph{Syst. Control Lett.}, vol. 154, p. 104985, 2021.

\bibitem{vanwaarde2021finsler}
H.~J. van Waarde and M.~K. Camlibel, ``A matrix {F}insler's lemma with
  applications to data-driven control,'' in \emph{Proc. 60\textsuperscript{th}
  IEEE Conf. Decis. Control (CDC)}, Austin, TX, USA, 2021, pp. 5770--5775.

\bibitem{schererweilandLMI}
C.~Scherer and S.~Weiland, ``Linear matrix inequalities in control,'' October
  2017, {DISC} lecture notes.

\bibitem{steentjes2020arxiv}
T.~R.~V. Steentjes, M.~Lazar, and P.~M.~J. {Van den Hof}, ``Distributed
  {$\mathscr{H}_2$} control for interconnected discrete-time systems: a
  dissipativity-based approach,'' \emph{ArXiv e-prints}, vol.
  arXiv:2001.04875v1, 2020.

\end{thebibliography}

\end{document}